\documentclass[a4paper,11pt]{article} 
\usepackage{syntonly}
\usepackage{pstricks,subfigure,epsfig}
\usepackage{amsmath,amsfonts,amssymb,amsthm,euscript,upgreek}
\usepackage{enumerate}
\usepackage{multirow}
\usepackage{appendix}

\newcommand{\R}{\mathbb{R}}                     

\newcommand{\CP}{\mathbb{C}\mathrm{P}}

\newcommand{\CH}{\mathbb{C}\mathrm{H}}

\newcommand{\po}{\forall}                       
\newcommand{\f}{\rightarrow}                    
\newcommand{\C}{\mathbb{C}}                     
\newcommand{\de}{\partial}                      

\newcommand{\M}{\mathrm{M}}

\newcommand{\g}{\mathrm{g}}
\newcommand{\ww}{\mathrm{\omega}}
\newcommand{\dd}{\mathrm{D}}
\newcommand{\Kj}{\mathrm{K}}
\newcommand{\N}{\mathrm{N}}

\newcommand{\K}{K\"{a}hler}
\newtheorem{theor}{Theorem}
\newtheorem{prop}[theor]{Proposition}

\newtheorem{remar}[theor]{Remark}

\title{\Large\bf  K\"ahler--Einstein  submanifolds of   the infinite dimensional  projective space}
\date{}
\author{\small Andrea Loi\\ \small Dipartimento di Matematica e Informatica
-- Universit\`{a} di Cagliari -- Italy\\ \small e-mail address: loi@unica.it\vspace{0.3cm} \\ \small and\vspace{0.3cm} \\ \small Michela Zedda\\ \small Dipartimento di Matematica e Informatica
-- Universit\`{a} di Cagliari -- Italy\\ \small e-mail address: michela.zedda@gmail.com }

\begin{document} 
\maketitle


\begin{abstract}
This paper consists of two main results.
In the first one  we describe all  
\K\ immersions of  a bounded symmetric domain 
into the  infinite  dimensional complex projective space in terms 
of the  Wallach set of the domain.  In the second one  we exhibit  an  example of  
complete  and  non-homogeneous    \K\--Einstein metric with negative scalar curvature which admits a  \K\ immersion  into the  infinite dimensional complex projective space.

\noindent
{\it{Keywords}}: K\"ahler metrics; bounded domains;
diastasis
function; symmetric space; complex space form; Wallach set.

\noindent
{\it{Subj.Class}}: 53C55, 58C25.

\end{abstract}
\section{Introduction and statement of the main results}
This paper deals
with holomorphic and  isometric (from now on  \K ) immersions
of complete noncompact \K\--Einstein manifolds into $\CP^{\infty}$, the
infinite dimensional complex projective space equipped with the Fubini--Study metric $g_{FS}$.
Throughout this paper if a \K\ manifold $(M, g)$ admits a \K\ immersion into $\CP^{\infty}$ then we will say either  that $(M, g)$ is a {\em \K\ submanifold} of $\CP^{\infty}$ or that $g$ is {\em projectively induced}.
The only known examples  of projectively induced \K\--Einstein metrics are the flat metric 
on the  complex Euclidean space ${\C}^n$  (see \cite{Cal}) and the Bergman metric on a  bounded homogeneous domains  (see \cite{kodomain}).
Hence,  it is natural to ask if there exists a complete  nonhomogeneous \K\--Einstein submanifolds of ${\CP^{\infty}}$.
The following theorem, which is the first result of this paper,
gives a positive answer to this question (see Section \ref{cartanhartogs} for the definition of Cartan--Hartogs domain).
 \begin{theor}\label{corw1}
There exists a continuous family of homothetic, complete,   nonhomogeneous and  projectively induced K\"ahler-Einstein metrics on each Cartan--Hartogs domain based on an irreducible bounded symmetric domain
of rank $r\neq 1$.
\end{theor}

Our result should be compared with the compact case. 
First,  it is an open problem   to classify the compact \K\--Einstein manifolds which admit a \K\ immersion into a finite dimensional complex projective space.
 Actually, the only known examples 
 of such manifolds are homogeneous and it is conjecturally true  these are the only ones
 (see e.g. \cite{note}, \cite{ch}, \cite{tak} and \cite{ts}).
Moreover, a family as in the previous theorem   cannot exist  in the compact case.
Indeed if  $cg$ are homothetic  \K\ metrics on a compact complex manifold $M$ such that 
$(M, c g), c\geq 1$ admits    a K\"ahler immersion into the finite dimensional complex projective space
 then, by simple topological reasons,  $c$ is forced  to be  a positive integer.

The proof of Theorem \ref{corw1} is based  on  recent results (see \cite{roos} and \cite{compl})
about  Einstein metrics on Cartan--Hartogs domains 
and on the following theorem which is  the second result of this paper  (see next section or \cite{arazy} for the definition of  the Wallach set of the domain $\Omega$).
\begin{theor}\label{wallach}
Let $\Omega$ be an irreducible bounded symmetric domain endowed with its Bergman metric $g_B$.  Then 
$(\Omega,c\g_B)$ admits a equivariant  K\"ahler immersion into $\CP^\infty$ if and only if $c\gamma$ belongs to $W(\Omega)\setminus \{0\}$, where $\gamma$ denotes the genus of $\Omega$.
\end{theor}


\vskip 0.3cm
The paper is organized as follows. In the next section we recall basic results on Calabi's diastasis function and Calabi's criterion for \K\ immersions into $\CP^{\infty}$. In Section \ref{sec3},
after describing Calabi's diastasis function for the Bergman metric of a bounded symmetric domain,   we prove  Theorem
\ref{wallach}. The last section is dedicated to  the proof 
of Theorem \ref{corw1}.

\vskip 0.3cm

\noindent 
The authors would like to thank Antonio J. Di Scala
for his very useful comments and remarks.

\section{The diastasis function and Calabi's criterion}\label{criterion}
In his seminal paper Calabi  \cite{Cal} (to whom we refer for details and further results) gives necessary and sufficient conditions for a $n$-dimensional  K\"ahler manifold $(M,\g)$ to admit a K\"ahler immersion into a complex space form. The key tool is the introduction of a very particular K\"ahler potential $\dd^M_p(z)$, that Calabi called \emph{diastasis}. Recall that a K\"ahler potential is a smooth function $\Phi$ defined in a neighbourhood of a point $p$ such that $\ww=\frac{i}{2}\bar\de\de\Phi$, where $\ww$ is the K\"ahler form associated to $\g$. In a complex coordinate system $(z)$ around $p$ one has
\begin{equation}
\g_{\alpha\bar\beta}=2\g\left(\frac{\de}{\de z_\alpha},\frac{\de}{\de \bar z_\beta}\right)=\frac{\de^2\Phi}{\de z_\alpha\de\bar z_\beta}.\nonumber
\end{equation}
A K\"ahler potential is not unique: it is defined up to an addition with the real part of a holomorphic function.
If $(M, \g)$ admits a 
K\"ahler immersion into a complex space form then $g$ is real analytic (see Theorem \ref{Calabi}
below). In this case by duplicating the variables $z$ and $\bar z$ a potential $\Phi$ can be complex analytically continued to a function $\tilde \Phi$ defined in a neighbourhood $U$ of the diagonal containing $(p,\bar p)\in M\times \bar M$ (here $\bar M$ denotes the manifold conjugated to $M$). The \emph{diastasis function} is the K\"ahler potential $\dd^M_p$ around $p$ defined by
\begin{equation}
\dd_p^M(q)=\tilde\Phi(q,\bar q)+\tilde\Phi(p,\bar p)-\tilde\Phi(p,\bar q)-\tilde\Phi(q,\bar p).\nonumber
\end{equation}
Observe that $\dd_p^M(q)$ is symmetric in $p$ and $q$ and $\dd_p^M(p)=0$.
The following theorem provides us with a  very useful characterization of the diastasis function.

\begin{theor}[characterization of the diastasis]\label{chardiast}
Among all the K\"ahler potentials the diastasis is characterized by the fact that in every coordinate system $(z)$ centered in $p$, the $\infty\times\infty$ matrix of coefficients $(a_{jk})$ in its power expansion around the origin
\begin{equation}\label{powexdiastc}
\dd^M_{p}(z,\bar z)=\sum_{j,k=0}^{\infty}a_{jk}(z)^{m_j}(\bar z)^{m_k},
\end{equation}
satisfy $a_{j0}=a_{0j}=0$ for every nonnegative integer $j$.
\end{theor}

Here we are using the following 
convention: we arrange every $n$-tuple of nonnegative
integers as the sequence  $m_j=(m_{j,1},\dots,m_{j,n})$ with non decreasing order, that is $m_0=(0,\dots,0)$ and if $|m_j|=\sum_{\alpha=1}^n m_{j,\alpha}$, $|m_j|\leq |m_{j+1}|$ for all positive integer $j$.
Further
$(z)^{m_j}$ denotes  the monomial in $n$ variables $\prod_{\alpha=1}^n z_\alpha^{m_{j,\alpha}}$.

The importance of the diastasis function for our purposes is expressed by 
the following three theorems due to  Calabi. Recall that  the Fubini-Study metric $g_{FS}$ on 
the infinite dimensional complex projective space  $\CP^\infty$ 
is this metric whose \K\ form $\omega_{FS}$ in homogeneous coordinates $Z_0,\dots Z_j,\dots $
is given by
$\omega_{FS}=\frac{i}{2}\de\bar\de\log (\sum_{j=0}^{\infty}|Z_j|^2)$.
 If $p_0=[1,0,0,\dots]\in U_0=\{Z_0\neq 0\}$,  then in the affine coordinates $z_j=\frac{Z_j}{Z_0}$ the diastasis 
$\dd_{p_0}^{\infty}:U_0\rightarrow \R$ around $p_0$ is 
given by 
$$\dd_{p_0}^{\infty}(z)=\log(1+\sum_{j=1}^\infty|z_j|^2).$$

\begin{theor}[hereditary property]\label{Calabi}
Let $f:(M,\g)\rightarrow \CP^{\infty}$ be a K\"ahler immersion
such that $f(p)=p_0$. Then the metric $\g$ is real analytic and 
$\dd_p^M=\dd_{p_0}^{\infty}\circ f :M\setminus f^{-1}(H_0)\rightarrow \R$,
where $H_0=\CP^{\infty}\setminus U_0$.
\end{theor}
\begin{theor}[Calabi's rigidity]\label{Calabirig}
Let $f_1, f_2 :(M, g)\rightarrow \CP^{\infty}$ be two full \K\ immersions.
Then there exists a unitary transformation $U$ of   $\CP^{\infty}$
such that $f_2=U\circ f_1$.
\end{theor}
Recall that a holomorphic map $f\!: (M,\g)\rightarrow \CP^\infty$ is said to be  {\em full} if $f(M)$ is not contained in any complex totally geodesic submanifold of $\CP^\infty$.
\begin{theor}[Calabi's criterion]\label{criterium}
A K\"ahler manifold $(M,\g)$ admits a local full K\"ahler immersion into $\CP^\infty$ if and only if the $\infty\times\infty$ matrix of coefficients $(b_{jk})$  in the power expansion
\begin{equation}\label{powexdiastcp}
e^{\dd_p^M(z,\bar z)}-1=\sum_{j,k=0}^{\infty}b_{jk}(z)^{m_j}(\bar z)^{m_k},
\end{equation}
is positive semidefinite of infinite rank. Moreover, if the manifold $M$ is assumed to be  simply connected a local K\"ahler immersion 
can be extended to  a global one.
\end{theor}



\section{The diastasis function of bounded symmetric domains and the proof of Theorem \ref{wallach}}\label{sec3}
In the following proposition we describe the diastasis of
a bounded symmetric domain 
and one of its important features which will be a key ingredient 
for the proof of our results (see also \cite{loidiast} for 
 the global aspects of the diastasis function
for these domains).
Recall that the Bergman metric  $\g_B$ is a K\"ahler metric on $\Omega$ whose associated K\"ahler form $\omega_B$ is given by $\omega_B=\frac{i}{2}\de\bar\de\log\Kj(z,z)$, where  $\Kj$ is the reproducing kernel for the Hilbert space of holomorphic $L^2$-functions on $\Omega$,
namely those $f\in Hol (\Omega)$ such that $\int_\Omega |f|^2d\mu (z)<\infty$, where $d\mu (z)$
is the standard Lebesgue measure on $\C^n$.
\begin{prop}\label{diastdom}
Let $\Omega$ be a bounded symmetric domain.
Then the diastasis for its Bergman metric $\g_B$ around the origin  is
\begin{equation}\label{diastberg}
\dd^\Omega_0(z)=\log(V(\Omega)\Kj(z,z)),
\end{equation}
where $V(\Omega)$ denotes 
the total volume of $\Omega$ with respect to the Euclidean measure of the ambient complex Euclidean space.
Moreover the matrix $(b_{jk})$ given by  (\ref{powexdiastcp}) for $\dd_0^\Omega$ satisfy $b_{jk}=0$ whenever $|m_j|\neq|m_k|$.
\end{prop}
\begin{proof}
$\dd^\Omega_0(z)$ is centered at the origin, in fact by the reproducing property of the  kernel we have
\begin{equation}
\frac{1}{ \Kj(0,0)}=\int_{\Omega} \frac{1}{\Kj(\zeta,0)}\Kj(\zeta,0)d\mu,\nonumber
\end{equation}
hence $\Kj(0,0)=1/V(\Omega)$, and substituting in (\ref{diastberg}) we obtain $\dd^\Omega_0(0)=0$. By the circularity of $\Omega$, that is $z\in \Omega,\ \theta\in\R$
 imply $e^{i\theta}z\in\Omega$,
rotations around the origin are automorphisms and hence isometries, that leave $\dd^\Omega_0$ invariant.  Thus we have $\dd^\Omega_0(z)=\dd^\Omega_0(e^{i\theta}z)$ for any $0\leq\theta\leq 2\pi$, that is, each time we have a monomial $(z)^{m_j}(\bar z)^{m_k}$ in $\dd^\Omega_0(z)$, we must have $$(z)^{m_j}(\bar z)^{m_k}=e^{i|m_j|\theta}(z)^{m_j} e^{-i|m_k|\theta}(\bar z)^{m_k}=(z)^{m_j}(\bar z)^{m_k} e^{(|m_j|-|m_k|)i\theta}$$ implying $|m_j|=|m_k|$.
This means that  every  monomial in the expansion of  $\dd^\Omega_0(z)$   has holomorphic and antiholomorphic part with the same degree. 
Hence, by Theorem \ref{chardiast}, $\dd^\Omega_0(z)$ is the diastasis
for $g_B$. By the chain rule the same property holds true for   $e^{\dd^\Omega_0(z)}-1$
and the second part of the proposition follows immediately. 
\end{proof}

Before proving Theorem \ref{wallach} we recall
the definition of the  {\em Wallach set}  of
an irreducible bounded symmetric domain  $\Omega$ of genus $\gamma$,
 referring  the reader to  \cite{arazy}, \cite{faraut} and \cite{upmeier} for more details and results.
 This set, denoted by $W(\Omega)$, consists of all
$\lambda\in\C$ such that there exists a Hilbert space
${\cal H}_\lambda$ whose   reproducing kernel is   $\Kj ^{\frac{\lambda}{\gamma}}$.
This is equivalent to the requirement that 
$\Kj ^{\frac{\lambda}{\gamma}}$
is positive definite, i.e.  for all $n$-uples of  points $x_1,\dots,x_n$ 
belonging to $\Omega$ the $n\times n$ matrix $(\Kj(x_{\alpha},x_{\beta})^{\frac{\lambda}{\gamma}})$, is positive  {\em semidefinite}.
It turns out (see Corollary $4.4$ p. $ 27$ in  \cite{arazy} and references therein)
that $W(\Omega)$ consists only of real numbers and depends   on two of the domain's invariants, denoted  by  $a$ (strictly positive real number) and  $r$
(the rank of $\Omega$).
More precisely we have
\begin{equation}\label{wallachset}
W(\Omega)=\left\{0,\,\frac{a}{2},\,2\frac{a}{2},\,\dots,\,(r-1)\frac{a}{2}\right\}\cup \left((r-1)\frac{a}{2},\,\infty\right).
\end{equation}
The set $W_d=\left\{0,\,\frac{a}{2},\,2\frac{a}{2},\,\dots,\,(r-1)\frac{a}{2}\right\}$ and the interval $W_c= \left((r-1)\frac{a}{2},\,\infty\right)$
are called respectively  the {\em discrete} and {\em continuous} part   of the Wallach set of the domain 
$\Omega$.
\begin{remar}\label{rchimm}\rm
When $\Omega$ has rank $r=1$, namely $\Omega$ is the complex hyperbolic space $\CH^d$,
then $g_B=(d+1)g_{hyp}$, where $g_{hyp}$ is the hyperbolic metric on $\CH^d$.
In this case (and only in this case) $W_d=\{0\}$ and $W_c=(0, \infty)$.
Therefore our Theorem \ref{wallach} is asserting that,
for all positive constants $c$, $(\CH^d, cg_{hyp})$ admits a full K\"ahler immersion into $ \CP^\infty$.
This is well-known and  can also be proved by using Calabi's criterion
(see  Theorem $13$ in \cite{Cal}).
\end{remar}

\begin{proof}[Proof of Theorem \ref{wallach}]
Let $f\!:(\Omega,c\g_B)\rightarrow\CP^\infty$ be a K\"ahler immersion, we want to show that 
$c\gamma$ belongs to  $W(\Omega)$, i.e. 
 $\Kj^c$ is positive definite.
Since $\Omega$ is contractible it is not hard to see that there exists a sequence
 $f_j, j=0, 1\dots$ of holomorphic functions defined on $\Omega$, not vanishing simultaneously, such that the immersion 
 $f$ is given by $f(z)=[\dots ,f_j(z), \dots ],\ j=0,1 \dots$, where $[\dots ,f_j(z), \dots ]$ denotes the equivalence
 class in $\ell ^2(\C)$ (two sequences are equivalent iff they differ by the multiplication by a nonzero complex number). Let $x_1,\dots,x_n\in \Omega$. Without loss of generality (up to unitary transformation of ${\C}P^{\infty}$)
we can assume that   $f(0)=e_1$, where $e_1$ is the first vector of the canonical basis of $\ell ^2({\C})$, and $f(x_j)\notin H_0$, $\po\ j=1,\dots, n$.
Therefore, by Theorem \ref{Calabi} and  Proposition \ref{diastdom},
we have 
$$c\dd^\Omega_0(z)=\log [V(\Omega)^c\,\Kj^c(z, z)]=\log\left(1+\sum_{j=1}^{\infty}\frac{|f_j(z)|^2}{|f_0(z)|^2}\right),\ \ z \in  \Omega\setminus f^{-1}(H_0).$$
\begin{equation}
V(\Omega)^c\,\Kj^c(x_{\alpha},x_{\beta})=1+\sum_{j=1}^\infty g_j(x_{\alpha})\overline{g_j(x_{\beta})},
\ g_k=\frac{f_k}{f_0}.\nonumber
\end{equation}
Thus for every $(v_1, \dots ,v_n)\in \C^n$ one has
$$\sum_{\alpha, \beta =1}^nv_{\alpha}\Kj ^c (x_{\alpha}, x_{\beta})\bar v_{\beta}=
\frac{1}{V(\Omega)^c}\sum_{k=0}^{\infty}|v_1g_k(x_1)+\cdots +v_ng_k(x_n)|^2\geq 0, g_0=1$$
and hence the matrix  $(\Kj^c(x_{\alpha}, x_{\beta}))$ is positive semidefinite.
Conversely, 
assume that $c\gamma\in W(\Omega)$.
Then, by the very definition of Wallach set,
there exists a Hilbert space ${\cal H}_{c\gamma}$ whose reproducing kernel
is $\Kj^c=\sum_{j=0}^{\infty}|f_j|^2$, where $f_j$ is an orthonormal basis of 
${\cal H}_{c\gamma}$.
Then  the holomorphic map 
$f :\Omega\rightarrow \ell^2 (\C)\subset \CP^{\infty}$  constructed by using this orthonormal basis 
satisfies $f^*(g_{FS})=c\g_B$. In order to prove that this map is equivariant  
write $\Omega =G/K$ where $G$ is the simple Lie group acting holomorphically and isometrically 
on $\Omega$ and $K$ is its isotropy group. For each $h\in G$ the map $f\circ h:(\Omega, c\g_B)\f\CP^{\infty}$ is a full \K\ immersion and therefore by Calabi's rigidity (Theorem \ref{Calabirig}) there exists
a unitary transformation $U_h$ of $\CP^{\infty}$ such that $f\circ h=U_h\circ f$ and we are done.
\end{proof}

\begin{remar}\rm
 In \cite{arazy} it is proven that  if $\lambda$ belongs to $W(\Omega)\setminus \{0\}$ then  $G$ admits a representation in the Hilbert space ${\cal H}_{\lambda}$. This is in accordance with our result. Indeed
if $c\gamma$ belongs to $W(\Omega)\setminus \{0\}$  then
the correspondence $h\mapsto U_h, h\in G$ defined in the last part of the  proof of Theorem \ref{wallach} is a representation of $G$.
\end{remar}

We conclude this section with some remarks regarding \K\ immersions of bounded symmetric domains into other complex space forms different from $\CP^{\infty}$.

\begin{remar}\rm
 The multiplication of the Bergman metric  by $c$ in  Theorem \ref{wallach} is harmless  when one studies the \K\ immersions of a \K\ manifold $(M, g)$  into the infinite dimensional complex Euclidean space  $\ell ^2({\C})$ equipped with the flat metric $g_0$. Indeed if $f:M\rightarrow \ell ^2({\C})$ satisfies $f^*(g_0)=g$ then 
$(\sqrt{c}f)^*(g_0)=cg$. It is worth pointing out that  the only bounded symmetric domain which admits a  \K\ immersion into $\ell ^2({\C})$ has rank one, i.e. it is the product of complex hyperbolic spaces (see    \cite{symm}  for a proof). Actually the authors believe the validity of the following conjecture:
{\em A complete \K\ manifold with negative scalar curvature which admits a \K\ immersion into $\ell ^2({\C})$ is a bounded symmetric domain of rank one}.
\end{remar}
\begin{remar}\rm
For the case  of  \K\ immersions of bounded symmetric domains  of noncompact type into the finite (resp. infinite) dimensional complex hyperbolic space $\CH^N$ (resp. $\CH^{\infty}$), namely the unit ball in $\C^n$ (resp. $\ell ^2({\C})$) equipped with the hyperbolic metric $g_{hyp}$, we can prove the following theorem:
 {\em if a $n$-dimensional bounded symmetric domain $(\Omega, c\g_B)$ admits a \K\ immersion into $\CH^N$ (resp. $\CH^{\infty}$)  then $\Omega =\CH^n$, $g_B=g_{hyp}$ and  $c=1$ (resp. $c\leq 1$).}
 The proof follows easily  by  Theorem 17 in \cite{alek} where it is shown that  a \K\ manifold which  admits a \K\ immersion into a  complex  hyperbolic space  is locally irreducible (for the  values of $c$ see   Theorem $13$ in \cite{Cal}).
 \end{remar}


\section{Cartan--Hartogs domains and the proof of Theorem \ref{corw1}}\label{cartanhartogs}
In order to prove Theorem \ref{corw1} we briefly recall some recent results
about Einstein metrics on Cartan--Hartogs domains.

Let $\Omega$ be an irreducible bounded symmetric domain of complex dimension $d$ and genus $\gamma$.
For all positive real numbers $\mu$ consider the family of Cartan-Hartogs domains 
\begin{equation}\label{defm}
\M_{\Omega}(\mu)=\left\{(z,w)\in \Omega\times\C,\ |w|^2<\N_\Omega^\mu(z,z)\right\},
\end{equation}
where $\N_\Omega(z,z)$ is the  {\em generic norm} of $\Omega$ namely,
$$N_{\Omega}(z, z)=(V(\Omega)K(z, z))^{-\frac{1}{\gamma}}.$$
The domain $\Omega$ is called  the {\em  base} of the Cartan--Hartogs domain 
$\M_{\Omega}(\mu)$ (one also  says that 
$\M_{\Omega}(\mu)$  is based on $\Omega$).
Consider on $\M_{\Omega}(\mu)$ the metric $\g(\mu)$  whose globally
defined  K\"ahler potential around the origin is given by
\begin{equation}\label{diastM}
\dd_0(z,w)=-\log(\N_{\Omega}^\mu(z,z)-|w|^2).
\end{equation}
The following theorem summarizes what we need about these domains
(see  \cite{roos} and \cite{compl} for a proof.)
\begin{theor}[G. Roos, A. Wang, W. Yin, L. Zhang, W. Zhang]\label{roos}
Let $\mu_0=\gamma/(d+1)$. Then $\left(\M_{\Omega}(\mu_0),\g(\mu_0)\right)$ is a complete K\"ahler--Einstein manifold which is homogeneous  if and only the rank of  $\Omega$ equals $1$,
i.e.  $\Omega = \CH^d$.
\end{theor}
\begin{remar}\rm\label{rrchimm}
Observe that when $\Omega=\CH^d$, we have $\mu_0=1$, $\M_\Omega(1)=\CH^{d+1}$ and $\g(1)=g_{hyp}$ (cfr. Remark \ref{rchimm}).
\end{remar}

In the following proposition, interesting on its own sake,   we describe the \K\ immersions of 
a Cartan--Hartogs domain 
into $\CP^{\infty}$
in terms of its base.

\begin{prop}\label{lemmadiastM}
The potential $\dd_0(z,w)$ given by (\ref{diastM}) is the diastasis
around the origin  of the metric
$\g(\mu)$.
Moreover,  $c\g(\mu)$ is projectively induced
if and only if  $(c+m)\frac{\mu}{\gamma}\g_B$ is 
projectively induced for every integer $m\geq 0$.
\end{prop}
\begin{proof}
The power expansion around the origin of $\dd_0(z,w)$ can be written as
\begin{equation}\label{powexpnot}
\dd_0(z,w)=\sum_{j,k=0}^\infty A_{jk} (zw)^{m_j}(\bar z\bar w)^{m_k}
\end{equation}
where $m_j$ are ordered $(d+1)$-uples of integer and $$(zw)^{m_j}=z_1^{m_{j,1}}\cdots z_d^{m_{j,d}} w^{m_{j,d+1}}.$$
In order to prove that   $\dd_0(z,w)$ is the diastasis for $\g(\mu)$
we need to verify  that $A_{j0}=A_{0j}=0$ (see Theorem \ref{chardiast}).
This is straightforward. Indeed if we take derivatives with respect 
either to $z$ or $\bar z$  is the same as deriving the function $-\log(\N_\Omega^\mu(z,z))=\frac{\mu}{\gamma}\dd_0^\Omega(z)$ that is the diastasis of $(\Omega,\frac{\mu}{\gamma}\g_B)$, thus we obtain $0$. If we take derivatives with respect either to $w$ or  $\bar w$  we obtain zero no matter how many times we derive with respect to $z$ or $\bar z$, since $\dd_0(z,w)$ is radial in $w$.

In order to prove the second part of the proposition
take the function 
\begin{equation}\label{funzione}
e^{c\dd_0(z,w)}-1=\frac{1}{(\N_\Omega^\mu(z,\bar z)-|w|^2)^c}-1,
\end{equation}
and using the same notations as in (\ref{powexpnot}) write the power expansion around the origin as
\begin{equation}\nonumber
e^{c\dd_0(z,w)}-1=\sum_{j,k=0}^\infty B_{jk} (zw)^{m_j}(\bar z\bar w)^{m_k}.
\end{equation}
By Calabi's criterion (Theorem (\ref{criterium})), $c\g(\mu)$
is projectively induced  if and only if $B=(B_{jk})$ is positive semidefinite of infinite rank.
The generic entry of $B$ is given by
\begin{equation}
B_{jk}=\frac{1}{m_j!\cdot m_k!}\frac{\de^{|m_j|+|m_k|}}{\de (zw)^{m_j}\de (\bar z\bar w)^{m_k}}\left(\frac{1}{(\N_\Omega^\mu(z,\bar z)-|w|^2)^c}-1\right)\Bigg|_0,\nonumber
\end{equation}
where $m_j! =m_{j,1}!\cdots m_{j,d+1}!$ and  $\de(zw)^{m_j}=\de z_1^{m_{j,1}}\cdots\de z_d^{m_{j,d}} \de w^{m_{j,d+1}}$.
By Proposition (\ref{diastdom}) we have
\begin{equation}\label{condi}
m_{j,1}+\cdots +m_{j,d}\neq m_{k,1}+\cdots +m_{k,d} \Longrightarrow B_{jk}=0,
\end{equation}
and since (\ref{funzione}) is radial in $w$ we also have
\begin{equation}\label{condii}
m_{j,d+1}\neq m_{k,d+1} \Longrightarrow B_{jk}=0.
\end{equation}

Thus, $B$ is a $\infty\times\infty$ matrix of the form
\begin{equation}
B=\left(
\begin{array}{cccccc}
0&0&0&0&0&0\\
0&E_1&0&0&0&\dots\\
0&0&E_2&0&0&\dots\\
0&\vdots&0&E_3&0&\dots\\
0&&\vdots&0&\ddots&
\end{array}
\right)\nonumber
\end{equation}
where the generic block $E_i$ contains derivatives
$\de(zw)^{m_j}$$\de(\bar z\bar w)^{m_k}$ of   order $2i$, $i=1,2,\dots$ such that $|m_j|=|m_k|=i$.
We can further write
\begin{equation}\label{matrixz}
E_i=\left(
\begin{array}{ccc}
F_{z(i)}(0)&0&0\\
0&F_{w(i)}(0)&0\\
0&0&F_{(z,w)(i)}(0)
\end{array}
\right)
\end{equation}
where $F_{z(i)}(0)$ (resp. $F_{w(i)}(0)$, $F_{(z,w)(i)}(0)$) 
contains derivatives
$\de(zw)^{m_j}$ \linebreak
 $\de(\bar z\bar w)^{m_k}$ (of   order $2i$  with $|m_j|=|m_k|=i$) such that $m_{j, d+1}=m_{k, d+1}=0$ (resp. $m_{j, d+1}=m_{k, d+1}=i$, $m_{j, d+1}, m_{k, d+1}\neq 0, i$).
(Notice also  that we have $0$ in all the other entries because of (\ref{condi}) and (\ref{condii})).
Since the derivatives are evaluated at the origin, deriving (\ref{funzione}) with respect to
$\de(zw)^{m_j}$$\de(\bar z\bar w)^{m_k}$ with $|m_j|=|m_k|=i$ and $m_{j, d+1}=m_{k, d+1}=0$
 is the same as deriving the function
\begin{equation}\label{funzdominio}
\frac{1}{(\N_\Omega^\mu(z,z))^c}-1=e^{c\frac{\mu}{\gamma}\dd_0^\Omega(z)}-1.
\end{equation}
Thus,  by Calabi's criterion, all the blocks $F_{z(i)}(0)$ are positive semidefinite if and only if $c\frac{\mu}{\gamma} \g_B$ is projectively induced.
Observe that  the blocks $F_{w(i)}(0)$ is semipositive definite without extras assumptions.  Indeed 
 if we consider derivatives $\de(zw)^{m_j}$$\de(\bar z\bar w)^{m_k}$ of (\ref{funzione})
with 
$|m_j|=|m_k|=i$ and $m_{j, d+1}=m_{k, d+1}=i$, since $\N_\Omega^\mu(z,z)$ evaluated in $0$ is equal to $1$, it is the same as deriving the function $1/(1-|w|^2)^c-1=\left(\sum_{j=0}^\infty |w|^{2j}\right)^c-1$
and the claim follows.
Finally,  consider  the block $F_{(z,w)(i)}(0)$. It can be written as
\begin{equation}
F_{(z,w)(i)}(0)=\left(
\begin{array}{cccc}
H_{z(i-1),w(1)}(0)&0&0&0\\
0&H_{z(i-2),w(2)}(0)&0&0\\
\vdots&&\ddots&\\
0&0&0&H_{z(1),w(i-1)}(0)
\end{array}
\right)\nonumber
\end{equation}
where the generic block $H_{z(i-m),w(m)}(0)$, $1\leq m\leq i-1$, 
contains derivatives
$\de(zw)^{m_j}$$\de(\bar z\bar w)^{m_k}$ of   order $2i$  such that $|m_j|=|m_k|=i$ and $m_{j, d+1}=m_{k, d+1}=m$ evaluated at zero
(as before, by (\ref{condi}) and (\ref{condii}) all entries outside these blocks are $0$).
Now it is not hard to verify   that  this blocks can be obtained by taking derivatives
$\de(zw)^{m_j}$$\de(\bar z\bar w)^{m_k}$ of order  $2(i-m)$ such that $|m_j|=|m_k|=2(i-m)$ and $m_{j, d+1}=m_{k, d+1}=0$  of  the function
\begin{equation}\nonumber
\frac{(m+c-1)!}{(c-1)!\; m!\ \N^{\mu(c+m)}_\Omega(z, z)}-1=
e^{(c+m)\frac{\mu}{\gamma}\dd_0^\Omega(z)}-1
\end{equation}
and evaluating at $z=\bar z =0$.
Thus, again by Calabi's criterion,
$F_{(z,w)(i)}(0)$
is positive semidefinite
iff $(c+m)\frac{\mu}{\gamma}g_B$, $m\geq 1$
is projectively induced
and this ends the proof of the proposition.
\end{proof}

We are now in the position to prove Theorem \ref{corw1}.

\begin{proof}[Proof of Theorem \ref{corw1}]
Take  $\mu=\mu_0=\gamma/(d+1)$ in (\ref{defm}) and $\Omega\neq\CH^d$. By Theorem \ref{roos} $\left(\M_\Omega(\mu_0),c\g(\mu_0)\right)$ is K\"ahler-Einstein, complete and nonhomogeneous for all positive real number $c$.
By Proposition \ref{lemmadiastM}
 $c\g(\mu_0)$  
 is projectively induced
if and only if $\frac{c+m}{d+1}\g_B$ is projectively induced, for all nonnegative integer $m$.
By Theorem \ref{wallach} this  happens if     $\frac{(c+m)}{d+1}\geq 
\frac{(r-1)a}{2\gamma}$.
Hence  $cg (\mu_0)$ with $c\geq \frac{(r-1)(d+1)a}{2\gamma}$ is the desired family of
projectively induced  K\"ahler-Einstein metrics.
\end{proof}

Notice that by choosing $0<c< \frac{a(d+1)}{2\gamma}$ (and $r\neq 1$)  one also gets the existence of  a continuous  family of complete,  nonhomogeneous and  K\"ahler-Einstein metrics which are not projectively induced.

\small{

}

\end{document}